\documentclass[12pt]{article}
\usepackage{amsfonts}
\usepackage{amssymb}
\usepackage{amsthm}
\usepackage[centertags]{amsmath}
\usepackage{newlfont}
\usepackage{graphicx}

\def\r{\mathbb R}

\newtheorem{theorem}{Theorem}[section]

\newtheorem{proposition}[theorem]{Proposition}
\newtheorem{remark}[theorem]{Remark}
\newtheorem{lemma}[theorem]{Lemma}
\newtheorem{corollary}[theorem]{Corollary}

\begin{document}
\title{Bifurcation of cylinders for wetting and dewetting models   with striped geometry}
\markright{Bifurcation of cylinders}
\author{Rafael L\'opez \footnote{Partially supported by MEC-FEDER grant no. MTM2011-22547 and
Junta de Andaluc\'{\i}a grant no. P09-FQM-5088.}\\
 Departamento de Geometr\'{\i}a y Topolog\'{\i}a\\
Universidad de Granada\\
18071 Granada. Spain\\
{\tt rcamino@ugr.es}}

\date{}
\maketitle

\begin{abstract} We show that some pieces of cylinders bounded by two parallel straight-lines  bifurcate in a family of periodic non-rotational surfaces with constant mean curvature and with the same boundary conditions. These cylinders are initial interfaces in a problem of microscale range modeling the morphologies that adopt a liquid deposited in a chemically structured substrate with striped geometry or a liquid contained in a right wedge with Dirichlet and capillary boundary condition on the edges of the wedge. Experiments show that starting from these cylinders and once reached  a certain stage,  the shape of the liquid changes drastically in an abrupt manner. Studying the stability of such cylinders, the paper provides a mathematical proof of the existence of these new interfaces obtained in experiments. The analysis is based on the theory of bifurcation by simple eigenvalues of Crandall-Rabinowitz.
 \end{abstract}



\section{Introduction and results}

This work is motivated by experiments realized in the Max Planck Institute of Colloids and Interfaces (MIPKG), at Potsdam,  on wetting and dewetting  of a liquid deposited on microchannels formed alternatively by   hydrophilic and hydrophobic strips (\cite{bl,ghll,kmlr,li}). See Fig. \ref{f-li}. In a microscopic scale and in absence of gravity, consider a long strip $\Omega$ contained in a plane $P$ such that $\Omega$ and $P-\Omega$ are made by different materials: $\Omega$ is  made by a hydrophilic substance whereas the substrate of $P-\Omega$ is hydrophobic. We place a droplet of water on top of $\Omega$  whose shape depends on the surface tension. Next, we add more liquid until that   touches the boundary of the strip and it starts to spread along it. Because $P-\Omega$ is hydrophobic, the liquid is forced to remain in the strip $\Omega$.  At the beginning, the liquid inherits the symmetries of the strip, that is, it is invariant in the non-bounded direction of $\Omega$, adopting cylindrical shapes. When we sufficiently increase the amount of liquid, there exists an instant where  the liquid suddenly exhibits bulges (\cite{bl,ghll,li}). See Fig. \ref{f-li}. In any stage, the liquid-air phase is modeled by a surface with constant mean curvature. Experimentally, this drastic transition between (pieces of) cylinders and new non-rotational morphologies motivates us to think in some type of non uniqueness results about the existence  of constant mean curvature surfaces   emanating from cylinders.

\begin{figure}[hbtp]
\begin{center}
  \includegraphics[width=.5\textwidth]{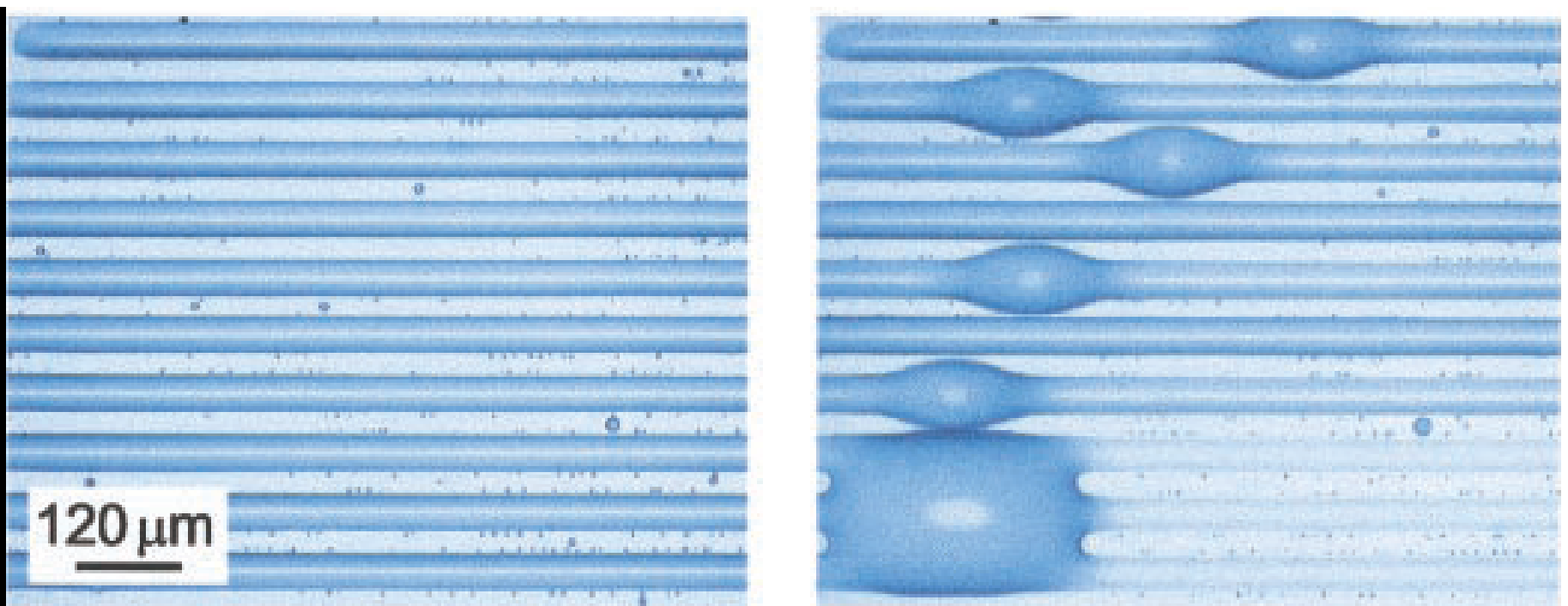}\hspace*{.5cm}\includegraphics[width=.4\textwidth]{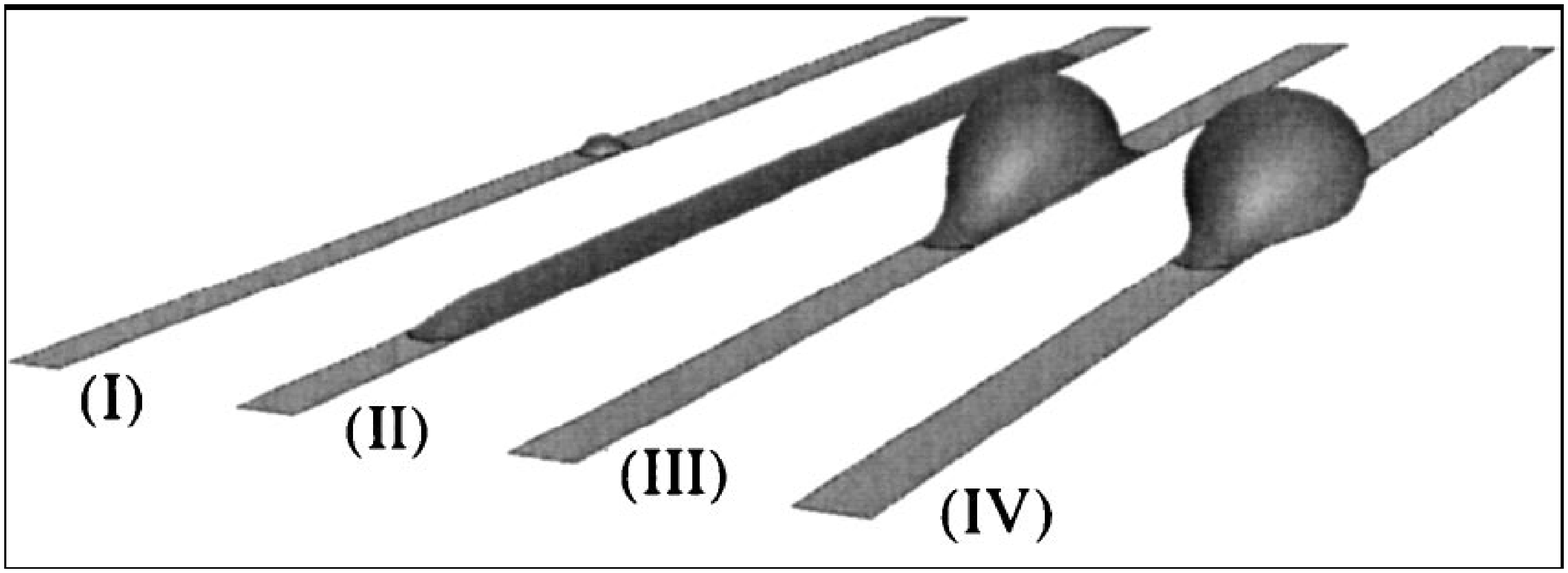}
\end{center}
\caption{Experiments and their graphic models obtained in MIPKG. On the left, it appears a planar domain chemically structured by strips made alternatively by hydrophilic and hydrophobic materials. In this picture, a sufficiently amount of liquid has been added in such way that the liquid covers the hydrophilic strips remaining pinned to the boundary lines. At the initial stages, the morphologies of the liquid are round cylinders. If we follow adding more liquid, experiments show that the cylinders become unstable and develop single bulges. In the right picture, there are graphic models developed in the MIPKG, showing the different geometric shapes. We   observe that  the surfaces of graphics (III) and (IV)    present symmetries  with respect to longitudinal   orthogonal planes. (Reprinted by courtesy of R. Lipowsky).}\label{f-li}
\end{figure}
  The second scenario in this article is the study of constant mean curvature surfaces in a wedge with
  Dirichlet and capillary conditions in each edge of the wedge, respectively. See Fig. \ref{f-ha}. Again, we focus in recent experiments  in melting processes realized in MIPKG  (\cite{kmlr}). Let a liquid be in a right angle wedge $W$ defined by two planes  $P_1\cup P_2$ and with axis $L=P_1\cap P_2$. Instead of $P_1$, we only consider an infinite strip $\tilde{P_1}\subset P_1$ of finite width being $L$ one of its boundary components. Let $\partial \tilde{P_1}=L\cup L_1$. One deposits a liquid droplet in $W$ close to  the axis $L$. We place  more liquid in such way that the liquid spreads in $W$ attaining  $L_1$ and we force that the liquid to be fixed in $L_1$, but that it can displace on $P_2$. In equilibrium, the first geometric configurations are pieces of circular cylinders, where one component of its boundary is $L_1$ and the other one moves freely on $P_2$, which it is a parallel straight-line $L_2$ to $L$. As we add more liquid, the boundary component of the free surface is pinned to $L_1$ (Dirichlet condition) whereas the other one remains in $P_2$ making contact angle (capillary condition). Experiments show that after a sufficiently amount of liquid deposited on the wedge, the cylindrical shape breaks its symmetries appearing bulges similarly as in the previous case: see  Fig. \ref{f-ha}, (IV). The new interfaces are surfaces $M$ with constant mean curvature included in the wedge with two boundary components, $\partial M=\Gamma_1\cup\Gamma_2$: the curve $\Gamma_1$ agrees with $L_1$ and the other component $\Gamma_2$ is a curve on the plane $P_2$ in such way that the interface $M$ makes constant angle with $P_2$ along $\Gamma_2$.

\begin{figure}[hbtp]
\begin{center}
  \includegraphics[width=.7\textwidth]{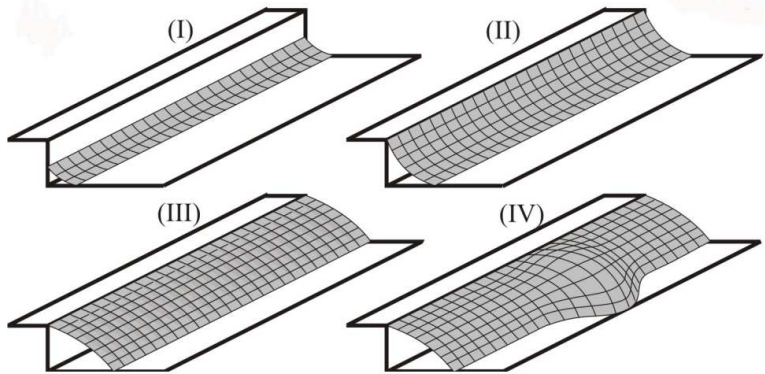}
\end{center}
\caption{Graphics models  obtained in MIPKG of a liquid deposited in a right wedge. At the initial stages and when the volume deposited is  small, the shape of liquid drop is cylindrical ((I) and (II)). Once reached the top in the vertical plane, the liquid begins  spreading on the horizontal plane making constant contact angle  (III). Here the cylinder is convex until that the added liquid is sufficiently big so the morphologies exhibit bulges (IV). (Reprinted by courtesy of R. Lipowsky).}\label{f-ha}
\end{figure}

In the above two settings, the first circular cylindrical liquids are stable under small perturbations of liquid. Stability implies uniqueness of morphologies in the sense that as we add liquid, the new surfaces obtained, which have constant mean curvature (possibly with different  values of mean curvature), are the only ones possible.  The authors analyze in \cite{bl} the bifurcation based on a number of numerical  diagrams relating the contact angle with the  volume of the liquid drop.

 In this article we realize a mathematical proof of such evidences using bifurcation theory. Exactly, we show:

\begin{theorem}\label{t1}
 Let   $\gamma\in (\pi/2,\pi)$. Consider the strip $\Omega=\{(x,y,0);-a\leq y\leq a\}$ and $\partial\Omega=L_1\cup L_2$. Denote $C(r,\gamma)$ a piece of a non-bounded cylinder of radius $r$ with boundary $\partial\Omega$ and making a contact angle $\gamma$ with $P$. Then there exists $T>0$, whose value is
 $$T=\frac{4\pi r\gamma}{\sqrt{4\gamma^2-\pi^2}},$$
 such that the cylinder $C(r,\gamma)$ bifurcates in a family of   non-rotational surfaces with constant mean curvature and  whose boundary is   $L_1\cup L_2$. These surfaces are periodic in the $x$-direction with period $T$.
\end{theorem}

\begin{theorem}\label{t2} Let $P_1$ and $P_2$ be two orthogonal planes, $W$ one of the quadrants determined by $P_1\cup P_2$ and $L=P_1\cap P_2$. Let $L_1\subset P_1$ be a straight-line parallel to $L$, $\gamma\in (0,\pi)$ and denote by $C(r,\gamma)$ a piece of a circular cylinder of radius $r$ included in $W$ bounded by two parallel straight-lines where one  is $L_1$, the other one lies in $P_2$ and the cylinder makes a contact angle $\gamma$ with
$P_2$.   Given a convex cylinder $C(r,\gamma)$, there exists $T>0$ such that
 the cylinder $C(r,\gamma)$ bifurcates in a family of   surfaces   with constant mean curvature contained in $W$ with two boundary components: one of them is $L_1$ and the other one lies in $P_2$ in such way that  the surfaces  make with  $P_2$ a contact angle $\gamma$ along this component. Moreover these surfaces are periodic in the direction of the axis of $W$ and the period is $T$.
\end{theorem}

Both results give us a new curve of solutions   as a parameter of the mean curvature. One branch is formed by appropriate pieces of cylinders with the same boundary conditions and the other one by the new surfaces that appear in above theorems. In our results, the new surfaces obtained by a bifurcation argument are not rotational because they contain two straight-lines and the only rotational surfaces with constant mean curvature including straight-lines are cylinders.

The existence of new surfaces  must occur when the stability of the cylinder fails. The first step in the analysis of a given bifurcation is to establish that a bifurcation has taken place. It is the case that if a known solution loses stability as a given parameter is varied. This is the reason that we previously need to give an analysis of stability of pieces of cylinders bounded by two prescribed straight-lines (first setting) or by a fixed straight-line and the other one moves in a plane (second setting). A similar situation occurs in the case that the boundary of the cylinder is empty, that is, as a complete surface. A recent argument of bifurcation shows that the classical Plateau-Rayleigh instability criterio of the cylinder (\cite{ra}) implies the existence of new periodic constant mean curvature surfaces originated by cylinders, which must be rotational, that is, Delaunay surfaces (\cite{ss}).

One of the first results on bifurcation of surfaces with constant mean curvature appeared in  \cite{vo}, where Vogel considered similar problems assuming cylinders in (non necessary right) wedges and whose two boundary components satisfy a capillary condition. Exactly, it is assumed that the contact angle with the edges of the wedge is constant and with the same value of angle and it was showed the existence of non-rotational configurations. Next, we point out the Ph. D. thesis of Patnaik (\cite{pa}) advised by Wente. In this work it is  considered the problem to find surfaces with minimum area enclosing a volume $V>0$ and whose boundary is formed by two prescribed coaxial circles in parallel planes. It is proved that for each $V$ there exists an area-minimizing surface, in particular, it is a surface with constant mean curvature. When the volume $V$ is small, the surface is rotationally symmetric, but if the volume of the surface increases until a critical volume,   new non-rotational surfaces are obtained  which  develop bulges again. Numerical graphics of such surfaces appear in \cite{ho}. More recently it has been studied problems of bifurcation in the theory of surfaces with constant mean curvature: \cite{ap,gh,jl,kpp,mp,ro}.   Special attention has received the bifurcation of (pieces of) nodoids (\cite{gh,kpp,mp,ro}).

In physics literature, the bifurcation from cylinders has been studied in \cite{bs} using an finite-element analysis.   In a more general context, studies on bifurcation have been realized for rotating liquid drops, that is, rigidly  liquid drops which rotate with constant angular velocity $\omega$ about an axis $L$. In this case, the interface is a surface whose mean curvature is a linear function on the square of the distance $d$ to the given axis $L$: $H=\kappa\omega^2 d^2+c$, for some real numbers $\kappa$ and $c$. An example of a such surface is a circular cylinder with axis $L$. If $r$ is the radius of the cylinder, its mean curvature $H$ satisfies $H=\kappa\omega^2 d^2+c$, with $\kappa\omega^2=1/(2r^3)$ and $c=0$. Because one can also choose $\omega=0$ and $c=1/(2r)$, a cylinder can be viewed as both a surface with constant mean curvature or the interface of a rotating liquid drop. Depending on the different assumptions on the boundary conditions, some authors have investigated the stability and bifurcation of rotating liquid drops from pieces of cylinders: \cite{kr,ks,kms,lm,ub}.

This article is organized as follows. In Section \ref{s2} we give the definition of stability of a surface with constant mean curvature. In Section \ref{s3} we study the stability of pieces of cylinders bounded by two given straight-lines which allows to show Theorem \ref{t1} in Section \ref{s4}. Next in Section \ref{s5} we analyze  the stability of pieces of cylinders in the second setting, showing Theorem \ref{t2} in Section \ref{s6}.

\section{Stability of surfaces with constant mean curvature}\label{s2}

In this section we recall some definitions and basic facts on the stability of constant mean curvature surfaces in Euclidean space. We refer to the reader to \cite{bce,rv}. Consider $\phi:M\rightarrow\r^3$ an immersion of a compact orientable surface $M$. A variation of $\phi$ is a differentiable map $\Phi:M\times(-\epsilon,\epsilon)\rightarrow\r^3$, $\epsilon>0$, such that $\phi_t:M\rightarrow\r^3$ defined by $\phi_t(p)=\Phi(p,t)$, $p\in M$ is an immersion for any $t\in (-\epsilon, \epsilon)$, and $\phi_0=\phi$. Associated with the variation $\Phi$, we define the area functional $A:(-\epsilon,\epsilon)\rightarrow\r$ by
$$A(t)=\int_M dA_t,$$
where $dA_t$ is the area element of $M$ with the induced metric by $\phi_t$, and the volume functional $V:(-\epsilon,\epsilon)\rightarrow\r$   by
$$V(t)=\int_{M\times[0,\epsilon]}\Phi^*(dV),$$
 where $\Phi^*(dV)$ is the pullback of the Euclidean volume element $dV$. The number $V(t)$ represents the signed volume enclosed between the surfaces $\phi$ and $\phi_t$.  The variation is called volume preserving if $V(t)=V(0)$ for all $t$.    The variational vector field of $\Phi$ is defined by
 $$\xi(p)=\frac{\partial\Phi}{\partial t}(p){\Big|}_{t=0}.$$
A variation $\Phi$ is called normal if $\xi=uN$ for some function $u$. We shall consider variations of $\phi$ that fix some components of $\partial M$ and the other ones, move in a given support. Because the two settings appeared in Introduction,  we consider surfaces whose boundary has two components $\Gamma_1$ and $\Gamma_2$.  Consider $\Pi$ an embedded connected surface in   $\r^3$ that divides the space into components and let us fix one of them, denoted by $W$. Let $\partial M=\Gamma_1\cup\Gamma_2$ be a decomposition into components, where $\Gamma_1$ is the part of the boundary that is pointwise fixed and $\Gamma_2$ the one that moves in the support $\Pi$. We say that $\Phi$ is an admissible variation of $\phi$ if $\phi_t(\mbox{int}(M))\subset W$, $\phi_t|_{\Gamma_1}=\phi|_{\Gamma_1}$  and $\phi_t(\Gamma_2)\subset\Pi$.

 Fix $\gamma\in (0,\pi)$.   Given an admissible variation $\Phi$,    the energy functional $E:(-\epsilon,\epsilon)\rightarrow\r$ is defined by $E(t)=A(t)-\cos\gamma S(t)$, where $S(t)$ is the area of the part $\Omega$ of $\Pi$ bounded by $\phi_t(\Gamma_2)$.   Let $N$ be a unit normal vector field along $\phi$ that points into the domain determined by $\phi(M)$ and $\Omega$ and let $\tilde{N}$ be the  unit normal vector to $\Pi$ pointing outside. Let $\nu$ (resp. $\bar{\nu}$) denote the unit exterior normal vectors to $\Gamma_2$ in $M$ (resp. in $\Omega$) and $H$ is the mean curvature of $\phi$. The first variation formulae for the energy $E$ and for the volume $V$ are
\begin{eqnarray*}
E'(0)&=&-2\int_M H u\ dM+\int_{\Gamma_2}\langle\xi,\nu-\cos\gamma\bar{\nu}\rangle\ ds\\
&=&-2\int_M H u\ dM+\int_{\Gamma_2}\langle\xi,\tilde{\nu}\rangle(\langle N,\tilde{N}\rangle-\cos\gamma)\ ds,\\
V'(0)&=&\int_M u\ dM,
\end{eqnarray*}
where $u=\langle N,\xi\rangle$ and $ds$ is the induced arc-length on $\partial M$. We say that the immersion $\phi$ is stationary if $A'(0)=0$ for any volume-preserving admissible variation of $\phi$. Using the above expression of $A'(0)$ and $V'(0)$, the immersion  $\phi$ is stationary is and only if $\phi$ has constant mean curvature and intersects $\Pi$ with constant  angle $\gamma$ along $\Gamma_2$, that is, $\langle N,\tilde{N}\rangle=\cos\gamma$ along $\Gamma_2$.

Denote by $\sigma$ and $\tilde{\sigma}$ the second fundamental form of $\phi:M\rightarrow\r^3$ and $\Pi\hookrightarrow \r^3$ with respect to $N$ and $-\tilde{N}$ respectively. For each smooth function $u$ on $M$ with $\int_M u\ dM=0$ there exists an admissible normal volume-preserving variation of $\phi$ with variational vector field $uN$. The second variation of $E$ is
$$E''(0)=-\int_M u(\Delta u+|\sigma|^2 u)\ dM+\int_{\Gamma_2}
u\Big(\frac{\partial u}{\partial\nu}-qu\Big) ds,$$
where
$$q=\frac{1}{\sin\gamma}\tilde{\sigma}(\tilde{\nu},\tilde{\nu})+\cot\gamma\sigma(\nu,\nu),$$
 $\Delta$ stands for  the Laplacian operator of $M$ induced by $\phi$ and $|\sigma|^2$ is the square of the norm of $\sigma$, which in terms of mean curvature $H$ and Gaussian curvature $K$ is $|\sigma|^2=4H^2-2K$. The immersion $\phi$ is called stable if $E''(0)\geq 0$ for all volume-preserving admissible normal variations of $\phi$. The second variation $E''(0)$ defines an index form $I$, which is a bilinear form on $H_0^1(M)$:
$$I(u,v)=\int_M (\langle\nabla u,\nabla v\rangle-|\sigma|^2 uv)\ dM-\int_{\Gamma_2}quv\ ds.$$
Here $H_0^1(M)$ is the first Sobolev space, that is, the completion of $C_0^\infty(M)$, $C_0^\infty(M)$ is the space of smooth functions on $M$ that vanish on $\Gamma_1$ and $\nabla$ means the gradient operator for the metric induced by $\phi$. Thus a stationary immersion is stable if and only if $I(u,u)\geq 0$ for all $u\in H_0^1(M)$.

The eigenvalue problem corresponding to the quadratic form $I$ is:
\begin{equation}
\left\{\begin{array}{lll}\label{eq-eigen}
& Lu+\lambda u=0& \mbox{on $M$}\\
&u=0& \mbox{on }\Gamma_1\\
& \dfrac{\partial u}{\partial \nu}-qu=0& \mbox{on $\Gamma_2$}
\end{array}
\right.
\end{equation}
where $L:H_0^1(M)\rightarrow L^2(M)$ is defined by $Lu=\Delta u+|\sigma|^2u$. The operator $L$ is the so-called  Jacobi operator. The next result is known (\cite{ch,ko}):
\begin{lemma}\label{le-eigen}
There exists a countable set of eigenvalues $\lambda_1<\lambda_2\leq\ldots$, with $\lambda_n\rightarrow +\infty$ as $n\rightarrow+\infty$. Moreover,
\begin{enumerate}
\item if $\lambda_1\geq 0$,   the immersion $\phi$ is stable.
\item if $\lambda_2<0$,   the immersion $\phi$ is unstable.
\end{enumerate}
Denote by $E_{\lambda}$ the vector subspace of the eigenfunctions of the eigenvalue $\lambda$ in \eqref{eq-eigen}. Then $L^2(M)=\bigoplus_{n=1}^\infty E_{\lambda_n}$.

\end{lemma}

\section{Stability of pieces of cylinders resting on a horizontal plane}\label{s3}

The Plateau-Rayleigh stability condition, experimented by Plateau, asserts that a cylinder of circular cross section of radius $r>0$ and bounded by two circles $h>0$ far apart  is stable if and only if $h< 2\pi r$ (\cite{ra}). In this section we consider the stability problem of a piece of a cylinder bounded by two straight-lines    resting in a  horizontal plane $P$. Some of computations that appear here are known in the literature. For example, the stability of surfaces of cylindrical geometry with capillary conditions and different settings was studied in \cite{vs} (see also references therein). See also a recent work  on the stability of cylinders focusing on the dynamics of the instability process (\cite{sl}).

Consider that $P$ is the plane of equation $z=0$, where $(x,y,z)$ are the usual coordinates of $\r^3$. Given $r>0$ and $\gamma\in (0,\pi)$, denote $C(r,\gamma)$ the piece of cylinder over $P$ whose boundary lies in $P$ and $C(r,\gamma)$ makes a contact angle $\gamma$ with $P$. This cylinder is described by
$$C(r,\gamma)=\{(x,y,z)-(0,0,r\cos\gamma)\in\r^3; y^2+z^2=r^2,  z\geq r\cos\gamma\}.$$

See Fig. \ref{fig3}. The boundary of this surface is formed by two  parallel straight-lines $L_1$ and $L_2$, namely,
$L_1\cup L_2=\{(x,\pm r\sin(\gamma),0);x\in\r\}$. This cylinder $C(r,\gamma)$  parametrizes as $\phi(t,s)=(t,r\cos(s),r\sin(s))-(0,0,r\cos\gamma)$ with  $s\in [\pi/2-\gamma,\pi/2+\gamma]$. If $\gamma=\pi/2$,  $C(r,\pi/2)$ is just a half-cylinder of radius $r$ resting on the plane $P$. The mean curvature of $C(r,\gamma)$ is  $H=1/(2r)$ computed with respect to the unit normal pointing to the convex domain bounded by $C(r,\gamma)$ and $P$. Denote $\Omega_\gamma=\{(x,y,0);-r\sin\gamma\leq y\leq r\sin\gamma\}\subset P$ the strip determined  by $\partial C(r,\gamma)$, with $\partial\Omega_\gamma=\partial
C(r,\gamma)=L_1\cup L_2$. Fix $W$ the upper half-space $z>0$. The normal $\tilde{N}$ of $P$ is $\tilde{N}=-(0,0,1)$. In this setting, and following the notation of Section \ref{s2}, we consider surfaces where the boundary is $\Gamma_1\cup\Gamma_2$, with $\Gamma_1=L_1\cup L_2$ and $\Gamma_2=\emptyset$, that is, we only have boundary conditions of Dirichlet type.

\begin{figure}[hbtp]
\begin{center}
  \includegraphics[width=.8\textwidth]{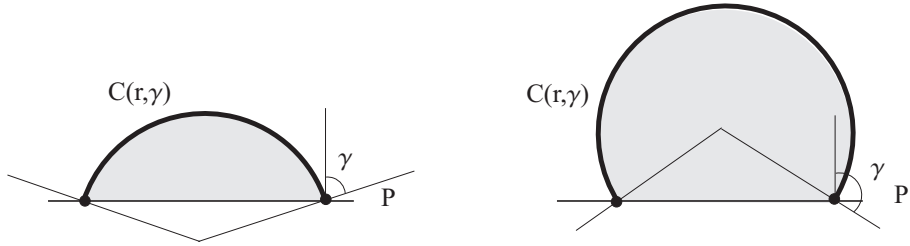}
\end{center}
\caption{Cross sections of cylinders resting on a horizontal plane $P$. On the left, the contact angle $\gamma$ satisfies $0<\gamma<\pi/2$; on the right we have $\pi/2<\gamma<\pi$.}\label{fig3}
\end{figure}

Because the cylinder $C(r,\gamma)$ is an unbounded surface, the stability of $C(r,\gamma)$ means stability for any compact subdomain of the cylinder. In our case, it is equivalent to consider the stability problem in truncated pieces $0\leq x\leq h$ of $C(r,\gamma)$ and to vary $h$. We consider the eigenvalue problem \eqref{eq-eigen} with $0$ as boundary data on $L_1\cup L_2$ and we use Lemma \ref{le-eigen}. We change $C(r,\gamma)$ by the rectangle $[0,h]\times [\pi/2-\gamma,\pi/2+\gamma]$ with variables $(t,s)$ and we use separation of variables.
Given a function $u=u(t,s)$,   we write $u$ as
\begin{equation}\label{uts}
u(t,s)=\sum_{n=1}^\infty g_n(s)\sin(\frac{n\pi}{h}t).
\end{equation}

As the function $u(t,s)$ vanishes in $s=\pi/2-\gamma$ and $s=\pi/2+\gamma$,   then $g_n(\pi/2-\gamma)=g_n(\pi/2+\gamma)=0$. We know the expression of $\Delta$ in cylindrical coordinates $(t,s)$ and because $K=0$, we have:
$$\Delta=\partial_{tt}+\frac{1}{r^2}\partial_{ss},\ \ |\sigma|^2=4H^2-2K=\frac{1}{r^2}.$$
In the eigenvalue problem \eqref{eq-eigen}, the first equation writes as
$$L(u)+\lambda u=\sum_{n=1}^\infty\Big(\frac{1}{r^2}g_n''(s)+(\frac{1}{r^2}-\frac{n^2\pi^2}{h^2}+\lambda)g_n(s)\Big)\sin(\frac{n\pi}{h}t).$$
Thus we have to solve
\begin{equation}\label{eq-g}
g_n''(s)+r^2\Big(\frac{1}{r^2}-\frac{n^2\pi^2}{h^2}+\lambda\Big)g_n(s)=0
\end{equation}
with boundary conditions
\begin{equation}\label{boundary1}
g_n(\frac{\pi}{2}-\gamma)=g_n(\frac{\pi}{2}+\gamma)=0.
\end{equation}
Set $C= r^2(\frac{1}{r^2}-\frac{n^2\pi^2}{h^2}+\lambda)$. We distinguish cases depending on the sign of $C$.
\begin{enumerate}
\item Case $C<0$. Let $c=\sqrt{-C}>0$.  The solution writes as $g_n(s)=A e^{cs}+B e^{-cs}$ for non-trivial constants $A$ and $B$. Equations \eqref{boundary1} are equivalent to
    $$Ae^{c(\frac{\pi}{2}-\gamma)}+Be^{-c(\frac{\pi}{2}-\gamma)}=Ae^{c(\frac{\pi}{2}+\gamma)}+Be^{-c(\frac{\pi}{2}+\gamma)}=0.$$
Combining both equations, we have $B^2=A^2e^{2c\pi}$ and
$$Ae^{c(\frac{\pi}{2}-\gamma)}(1\pm e^{2c\gamma})=0.$$
Then  $2\gamma=0$, which it is impossible.
\item Case $C=0$. Then $g_n(s)=As+B$, $A,B\in\r$. The boundary conditions \eqref{boundary1} give immediately a contradiction.
\item Case $C>0$. Let $c=\sqrt{C}>0$. Now $g_n(s)=A\cos(cs)+B\sin(cs)$, where $A,B\in\r$. The boundary conditions \eqref{boundary1} write respectively as
$$A\cos(c(\frac{\pi}{2}-\gamma))+B\sin(c(\frac{\pi}{2}-\gamma))=0.$$
$$A\cos(c(\frac{\pi}{2}+\gamma))+B\sin(c(\frac{\pi}{2}+\gamma))=0.$$
From the first equation we have $A=-\tan(c(\frac{\pi}{2}-\gamma))B$. Putting in the second one,
$\tan(c\pi)=\tan(c(\pi-2\gamma))$. This means that there exists $k\in\mathbb{Z}$ such that $c(\pi-2\gamma)=c\pi+k\pi$. Thus, there are non-trivial solutions $g_n$ of \eqref{eq-g} for some $n\in\mathbb{N}$ if and only if
$$c= \frac{k\pi}{2\gamma}$$
for some $k\in\mathbb{N}$ because $c>0$. From the value of $C$, we obtain explicitly all the eigenvalues of \eqref{eq-eigen}:
\begin{equation}\label{kn}
\lambda_{k,n}=\frac{1}{r^2}\Big(\frac{k^2\pi^2}{4\gamma^2}-1\Big)+\frac{n^2\pi^2}{h^2}.
\end{equation}
\end{enumerate}

We conclude:

\begin{proposition}\label{pr-stable}
\begin{enumerate}
\item If $\gamma\in (0,\pi/2]$, the  cylinder $C(r,\gamma)$ is stable.
\item Assume $\gamma\in(\pi/2,\pi)$. Consider a cylinder $C(r,\gamma)$ of length $h$. Then $\lambda_1\geq 0$  if and only if $h\leq h_0$, where
\begin{equation}\label{h-li}
h_0= \frac{2\pi r\gamma}{\sqrt{4\gamma^2-\pi^2}}.
\end{equation}
In such case, the surface is stable.
\item A cylinder $C(r,\gamma)$ with $\gamma\in(\pi/2,\pi)$ is unstable.
\end{enumerate}
\end{proposition}

\begin{proof} If $\gamma\in (0,\pi/2]$ and from \eqref{kn}, we have $\lambda_{k,n}\geq 0$ for any $h$. Then Lemma \ref{le-eigen} implies that $C(r,\gamma)$ is stable. If $\gamma\in(\pi/2,\pi)$,  we know from \eqref{kn} that the first eigenvalue   corresponds with $\lambda_{1,1}$. Then $\lambda_{1,1}\geq 0$ if and only if $h\leq h_0$ and Lemma \ref{le-eigen} implies stability. Moreover, if $\gamma>\pi/2$ and if $h$ is sufficiently big, the value of $\lambda_{k,n}$ in \eqref{kn} is negative for  many values of $k$ and $n$. Then Lemma \ref{le-eigen} assures that $C(r,\gamma)$ is unstable.
\end{proof}

\section{Proof of Theorem \ref{t1}}\label{s4}

The proof uses the  standard theory for bifurcation problems with a one-dimensional null space of Crandall and Rabinowitz (\cite{cr}). Let $\phi:M\rightarrow\r^3$ be an immersion with constant mean curvature $H_0$.  Let $V$ be an open of $0\in C_{0}^{2,\alpha}(M)$ such that for any $u\in V$, the normal graph $\phi_u:M\rightarrow\r^3$ defined by $\phi_u=\phi+uN$ is an immersion. Denote $H(u)$ the mean curvature of $\phi_u$ and define the map $F:V\times\r\rightarrow C^{\alpha}(M)$ by
$$F(u,H)=2(H-H(u)).$$
We see that $F(0,H_0)=0$. Moreover, the immersion $\phi_u$ has constant mean curvature if and only if there exists $H\in\r$ such that
\begin{equation}\label{eq-f}
F(u,H)=0.
\end{equation}

The next result is known in the literature (for example, \cite{ka,ko,vo}):

\begin{lemma} The functional $F$ is Fr\'echet differentiable with respect $u$ and $H$. The partial with respect to the first variable $u$ is
$$D_u F(0,H)v=-L(v),   v\in C^2_0(M),$$
where $L$ is the Jacobi operator.
\end{lemma}

We also need the next result about the solvability of the equation $\lambda u-L(u)=f$ (\cite{ko}):

\begin{lemma}\label{le-ko} Let $\phi:M\rightarrow\r^3$ be an immersion. Given $\lambda\in\r$ and  $f\in L^2(M)$, we consider the equation
$$\lambda u-L(u)=f,\ u\in H^1_0(M).$$
\begin{enumerate}
\item If $\lambda$ is not an eigenvalue of \eqref{eq-eigen}, there is a unique solution.
\item If $\lambda$ is an eigenvalue of \eqref{eq-eigen}, there is a solution if and only $f$ is $L^2$-orthogonal to $E_\lambda$.
\end{enumerate}
\end{lemma}
The uniqueness problem of solutions of \eqref{eq-f} is related with the Implicit Function Theorem and the solutions of the Jacobi equation $\Delta u+|\sigma|^2u=0$.  If  $D_uF(0,H_0):C_0^{2,\alpha}(M)\rightarrow C^\alpha(M)$ is bijective, there exists $\delta>0$ and a unique map $\varphi:(H_0-\delta,H_0+\delta)\rightarrow C_0^{2,\alpha}(M)$ such that $\varphi(H_0)=0$ and $F(\varphi(H),H)=0$ for any $|H-H_0|<\delta$. In such case, the immersion defined by $\phi+\varphi(H)N$ has constant mean curvature $H$.

On the other hand, assume that $\lambda=0$ is not an eigenvalue of the problem \eqref{eq-eigen}, that is, the only solutions of the Jacobi equation are trivial. This means that $D_uF(0,H_0)$ is one-to-one. Indeed,   $D_u F(0,H_0)$ is injective: if $v\in C_0^2(M)$ satisfies
$D_uF(0,H_0)(v)=0$, that is, $Lv=0$,   the solution is unique by using Lemma \ref{le-ko}. Then necessarily $v=0$. On the other hand, $D_uF(0,H_0)$ is surjective because given $f\in  L^2(M)$,  equation  $D_uF(0,H)(v)=f$ has a   solution
by   Lemma \ref{le-ko} again. Thus, $D_uF(0,H)$ is one-to-one and the Implicit Function Theorem yields the result.

In the case that $\lambda=0$ is an eigenvalue of \eqref{eq-eigen}, we can apply the Implicit Function Theorem in the next particular case (\cite[Lemma 3.3]{ko}):

\begin{lemma}\label{le-43} Let $\phi:M\rightarrow\r^3$ be an immersion with constant mean curvature $H_0$. Assume that $\lambda=0$ is an eigenvalue of \eqref{eq-eigen} with $E_0=\mbox{span}\{u_0\}$ and $\int_M u_0\ dM\not=0$. Then there exits an open $V$ around $0$ and a unique injective map $\psi:V\rightarrow C^{2,\alpha}_0(M)$, $\psi(H_0)=0$, such that for any $u\in V$, $\phi+(u+\psi(H))N$ has constant mean curvature $H$ with the same boundary as $\phi$. Moreover, there exists no other immersion on $M$ of constant mean curvature with the same boundary than $\phi$. In particular, this happens if $\lambda_1=0$.
\end{lemma}

By Proposition \ref{pr-stable}, we know that for a cylindrical channel $C(r,\gamma)$ with $0<\gamma\leq \pi/2$,  the first eigenvalue of \eqref{eq-eigen} is non-negative. Then the Implicit Function Theorem implies that  there exists a unique deformation of $C(r,\gamma)$  by surfaces  with constant mean curvature with the same boundary $\partial C(r,\gamma)$. Indeed, the surfaces of the deformation is given by pieces of cylinders again  which are described by  $\{M_t;|t|<\epsilon\}$, where $M_t=C(r\frac{\sin\gamma}{\sin(\gamma+t)}, \gamma+t)$, $|t|<\epsilon$ for $\epsilon>0$ sufficiently small.

Therefore, and in order to find a point of bifurcation,
  we have to pay our attention in those cylinders $C(r,\gamma)$ with $\gamma>\pi/2$.
It follows from the stability analysis given in Section \ref{s3} that the cylindrical channel $C(r,\gamma)$ for a wavelength $h$, $0<h\leq h_0$,
where $h_0$ is the value defined in \eqref{h-li}, is stable because the eigenvalues are all non-negative. If $h\in (h_0,2h_0)$, the smallest eigenvalue is negative but the other $\lambda_{k,n}$ are all positive  until that we reach the value $h=2h_0$, where the second eigenvalue is zero. If we fix the boundary of all $C(r,\gamma)$ to be the straight-lines $L_1\cup L_2$ and consider the value of the radius  $r$ as a variable parameter (or equivalently, the value of the mean curvature $H$ of the cylinder $C(r,\gamma)$), then the above statement may be interpreted as saying that the cylindrical solution  loses stability as the parameter $r$ increases through the critical value $h=2h_0$, that is,
$$T=2h_0=\frac{4\pi r\gamma}{\sqrt{4\gamma^2-\pi^2}}.$$
One expects that at a point where a known curve of solutions loses stability, a new branch of solutions bifurcates from the known curve. In our case, we regard the mean curvature $H$ as a bifurcation parameter and we want to show that when  $h=2h_0$, a family of non-rotational constant mean curvature surfaces and with boundary $L_1\cup L_2$ bifurcates off the family of cylindrical channels $C(r,\gamma)$. Then we are looking for solutions of the equation $F(u,H)=0$ in a neighborhood of the solution $(u,H)=(0,H_0)$ representing a piece of a cylindrical channel with radius $r=1/(2H_0)$.

The result that we shall apply is the bifurcation from a simple eigenvalue theorem of Crandall and Rabinowith, which we recall now in our context:

\begin{theorem}[\cite{cr}]\label{t-cr} Let $F:X\times I\rightarrow Y$ be a twice continuously Fr\'echet differentiable functional, where $X$ and $Y$ are Banach spaces, $I\subset \r$ and $H_0\in I$. Suppose $F(0,H)=0$ for all $H\in I$ and
\begin{enumerate}
\item $\mbox{dim Ker}(D_u F(0,H_0))=1$. Assume that Ker$(D_u F(0,H_0))$ is spanned by $u_0$.
\item The codimension of the range of $D_u F(0,H_0) $ is $1$, i.e., $F(0,H_0)$ is a Fredholm operator of index zero.
\item $D_HD_uF(0,H_0)(u_0)\not\in\mbox{rank }D_u F(0,H_0) $.
\end{enumerate}
Then  there exists a nontrivial  continuously differentiable curve through $(0,H_0)$, namely  $(u(s),H(s))$, $s\in(-\epsilon,\epsilon)$ with $u(0)=0$, $H(0)=H_0$, such that $F(u(s),H(s))=0$ for any $|s|<\epsilon$. Moreover, $(0,H_0)$ is a bifurcation point of the equation $F(u,H)=0$ in the following sense: in a neighborhood  of $(0,H_0)$ the set of solutions of $F(u,H)=0$ consists only of the curve $(0,H)$ and the curve $(u(s),H(s))$.
\end{theorem}

Here we take $X=V\subset C^{2,\alpha}_0(M)$ and $Y=C^\alpha(M)$. Fix a radius $r>0$ (or a value of the mean curvature $H_0=1/(2r)$). Consider a cylindrical channel $C(r,\gamma)$ under the hypothesis of Theorem \ref{t1} with length $T$ given by the above value. In order to apply the Crandall-Rabinowitz scheme, we seek non trivial solutions of \eqref{eq-eigen} that are $T$-periodic in the $x$-direction for some period $T>0$. We use separation of variables as in Section \ref{s3}. Thus, given a function $u$ on $C(r,\gamma)$ we consider $u$ defined in $\r/2\pi T\mathbb{Z}\times [\frac{\pi}{2}-\gamma,\frac{\pi}{2}+\gamma]$ and we write $u$ as a Fourier expansion on the functions $\sin (2\pi n t/T)$ and $\cos(2\pi n t/T)$. As we have looking for eigenvalues of the periodic  problem \eqref{eq-eigen} in the $t$-variable, the function $\cos(2\pi nt/T)$ writes as $\sin(2\pi n t /T+\tilde{h})$ for appropriate constant $\tilde{h}$, which does not affect to our problem. Then we can write $u$ in the following way
\begin{equation}\label{2uts}
u(t,s)=\sum_{n=1}^\infty g_n(s)\sin(\frac{2\pi n}{T}t).
\end{equation}
Using the expression of the operator $L$ in cylindrical coordinates, the functions $g_n$ satisfy $g_n''(s)+c^2 g_n(s)=0$ with
$$c^2=r^2\Big(\frac{1}{r^2}-\frac{4n^2\pi^2}{T^2}+\lambda\Big).$$
The solutions of $g_n$ are,  up constants,
$$g_n(s)=\sin(\frac{k\pi(s-(\frac{\pi}{2}-\gamma))}{2\gamma}),\ \ k\in\mathbb{N}.$$
Denote for $k,n$ the eigenfunctions
$$u_{k,n}(t,s)=\sin{(\frac{k\pi(s-(\frac{\pi}{2}-\gamma))}{2\gamma})}\sin{(\frac{2\pi n}{T}t)}, \ (t,s)\in\frac{\r}{2\pi T\mathbb{Z}}\times[\frac{\pi}{2}-\gamma,\frac{\pi}{2}+\gamma]$$
whose eigenvalues are
$$\lambda_{k,n}=\frac{1}{r^2}\Big(\frac{(k^2-n^2)\pi^2+4n^2\gamma^2}{4\gamma^2}-1\Big).$$
Then $0$ is an eigenvalue for $k=n=1$, that is, $\lambda_{1,1}$.  The eigenspace $E_0$ for the zero eigenvalue is spanned by $u_{1,1}$:
\begin{equation}\label{u11}
E_0=\mbox{span}\{u_{1,1}\}=\mbox{span}\{\sin{(\frac{\pi(s-(\frac{\pi}{2}-\gamma))}{2\gamma})}\sin{(\frac{2\pi}{T}t)}\}.
\end{equation}
In particular, $\mbox{dim}(E_0)=1$. In order to have the range of $L(u_{1,1})$, we calculate Im$(L(u_{1,1}))$. Let $f\in \mbox{Im}(L(u_{1,1}))$. Then there is $v$ such that $L(v)=f$. Since $0$ is an eigenvalue of $L$, by Lemma \ref{le-ko}, item 2, the necessary and sufficient condition is that $\int_M u_{1,1}v\ dM=0$ for any $v\in \mbox{Ker}(L)$. As dim(Ker($L$))$=1$, this means that the image of $L$ is the orthogonal subspace of $u_{1,1}$, $E_0^\bot$, showing that the codimension of $\mbox{rank}D_u F(0,H_0)$ is $1$.

Finally, we have to show that $D_HD_u F(0,H_0)(u_{1,1})\not\in \mbox{Im}(D_uF(0,H_0))$. We compute the partial of $D_uF(0,H_0)$ with respect to the variable $H$. We point out that in our result on bifurcation, the mean curvature is a parameter. In our case, given a cylinder $C(r,\gamma)$, $r=1/(2H)$ and
$$D_uF(0,H)(v)=L(v)=v_{uu}+4H^2 v_{ss}+4H^2v.$$
Thus
\begin{equation}\label{duh}
D_HD_u F(0,H)(v)=8H(v_{ss}+v).
\end{equation}
Replacing into \eqref{duh}   the expression of $u_{1,1}$ given in \eqref{u11}, we obtain
\begin{eqnarray*}
D_HD_uF(0,H_0)(u_{1,1})&=&8H_0(1-\frac{\pi^2}{4\gamma^2})\Big(\sin(\frac{\pi(s-(\frac{\pi}{2}-\gamma))}{2\gamma}\sin(\frac{2\pi t}{T})\Big)\\
&=& 8H_0(1-\frac{\pi^2}{4\gamma^2})u_{1,1}.
\end{eqnarray*}
We suppose that there exists $v$ such that $L(v)=D_HD_uF(0,H)(u_{1,1})$. Then using Lemma \ref{le-ko}, we have
\begin{equation}\label{in-u11}
\int_M u_{1,1}D_HD_uF(0,H_0)(u_{1,1})\ dM=0.
\end{equation}

Thus \eqref{in-u11} writes as
$$\int_M 8H_0(1-\frac{\pi^2}{4\gamma^2})u_{1,1}^2\ dM=0,$$
which it is a contradiction because $\gamma\not=\pm\pi/2$. This shows our assertion and it ends the proof of Theorem \ref{t1}.

The surfaces obtained in Theorem \ref{t1} and close to the value $H_0$, are embedded, periodic with period $T$ and lie in one side of $P$. The fact that the mean curvature is constant and the periodicity allow to know something more about the geometry of the new surfaces obtained by the bifurcation theory.

\begin{corollary}\label{co} Let $\Omega$ be a strip in a plane $P$ and denote   $Q$ the orthogonal plane to $P$ parallel to $\partial\Omega$   that  divides $\Omega$ in two symmetric domains. Consider  an embedded surface $M$ with constant mean curvature spanning $\partial\Omega$ and  periodic in the direction of $\partial\Omega$. If $M$ lies in one side of $P$, then $M$ is symmetric with respect to $Q$.
\end{corollary}

\begin{proof}
The proof uses in a standard way the Alexandrov reflection method by a uniparametric family of parallel planes $Q_t$ to $Q$ that foliate $\r^3$ (\cite{al}). For this, we take the $3$-domain $W$ bounded by $P$ and $\Omega$ which it is possible because $M$ is embedded and $M$ lies over $P$. Assume that $P$ is the plane $z=0$, $\Omega=\{(x,y)\subset\r^2;-m\leq y\leq m\}$ and $M$ is included in the halfspace $z> 0$. By the periodicity of the surface, $M$ is bounded along the $y$-direction.

Let $Q_t$ be the plane $y=t$ so that $Q_0=Q$. We introduce the next notation. If $A\subset\r^3$ is a subset of Euclidean space, let $A_{t}^{+}=A\cap \{(x,y,z)\in\r^3; y>t\}$,  $A_{t}^{-}=A\cap \{(x,y,z)\in\r^3; y<t\}$ and $A^*$ the reflection of $A$ with respect to $Q_t$. Starting from $t=+\infty$, the boundedness of  $M$ along the $y$-direction assures that the planes $Q_t$ do not touch $M$ if $t$ is sufficiently big. We decrease $t$ until the first time $t=t_0\geq m$ such that $Q_{t_0}$ touches $M$. Let us follow  $t\searrow 0$. For values $t<t_0$ and close to $t_0$, the surface $(M_t^{+})^*$ lies included in the domain $W$, that is, $(M_t^{+})^*\subset W$ . We continue with the process until that this property of inclusion fails the first time at $t=t_1$, $0\leq t_1<t_0$. In such case, we have two possibilities:
\begin{enumerate}
\item There is a common tangent point between $(M_{t_1}^{+})^*$ and $M_{t_1}^{-}$. The maximum principle of the constant mean curvature equation implies that both surfaces  agree, that is, $(M_{t_1}^{+})^* = M_{t_1}^{-}$ (\cite{gt}). Then  $Q_{t_1}$ is a plane of symmetry of $M$. In particular, $Q_{t_1}$ is a plane of  symmetry of the boundary of $M$, namely, $\partial M=\partial\Omega$, which it means that  $t_1=0$. This  proves  the result.
\item The value $t_1$ is $0$ and there is not a common tangent point between $(M_0^{+})^*$ and $M_0^{-}$. Then we start with the Alexandrov process with values $t$ close to $t=-\infty$ and consider the reflections of $M_{t}^-$ across $Q_t$, that is, $(M_t^{-})^*$. Using the fact that $t_1=0$ and that $(M_0^{+})^*$ and $M_0^{-}$ have not tangent point,  necessarily there exists $t_2<0$ such that $(M_{t_2}^{-})^*$ has a tangent point with $M_{t_2}^+$. The maximum principle would imply that the plane $Q_{t_2}$ is a plane of symmetry of $M$, which it is a contradiction because $\partial\Omega$ is not symmetric with respect to $Q_{t_2}$. This shows that this case is impossible.
\end{enumerate}

\end{proof}

As a consequence of Corollary \ref{co}, the surfaces obtained in Theorem \ref{t1} and close to the bifurcation point inherit the longitudinal symmetries of $\Omega$, that is, they are invariant by the symmetries with respect to the longitudinal plane that is orthogonal to $P$. This gives a mathematical support about the experiments and graphic models that appeared in Fig. \ref{f-li}.

\section{Stability of pieces of cylinders in right wedges}\label{s5}

Consider a wedge $W$ of angle $\pi/2$ and denote $P_1$ and $P_2$ the two half-planes that define $W$ with $L:=P_1\cap P_2$ the axis of the wedge. We study the stability of a cylinder bounded by two parallel straight-lines $L_1\cup L_2$, one of them, namely $L_1$, is prescribed in $P_1$ and parallel to $L$ and the other one, $L_2$, moves on $P_2$. Denote $\gamma\in (0,\pi)$ the angle that makes the cylinder with the plane $P_2$ along $L_2$ and  $C(r,\gamma)$ the corresponding cylinder. We assume that $L$ is the $x$-axis,  $P_1$ is the plane $z=0$,  $P_2$ is the plane $y=0$ and $W$ is the quadrant $y,z>0$. We parametrize the cylinder  $C(r,\gamma)$ by $\phi(t,s)=(t,r\cos(s),r\sin(s))$ with $s\in [0,\beta]$, $\beta\in(0,3\pi/2)$. As in Section \ref{s3}, it is enough to focus for truncated cylinders of length $h>0$. Let us take a cylinder of length $h$ by letting $0\leq x\leq h$. The eigenvalue problem corresponding to the quadratic form $I$ is given by \eqref{eq-eigen} where now $\Gamma_1=L_1$ and $\Gamma_2=L_2$. We use separation of variables again and
define the function $u=u(t,s)$ by
\begin{equation}\label{uts2}
u(t,s)=\sum_{n=1}^\infty g_n(s)\sin(\frac{n\pi}{h}t),
\end{equation}
where $s\in [0,\beta]$, $0\leq t\leq h$. The boundary conditions are
$$u(t,0)=0, \ \ \ \frac{\partial u}{\partial \nu}(t,\beta)-qu(t,\beta)=0.$$
Here
$$\nu(t,\beta)=\frac{1}{r}\phi_s(t,\beta),\ \ \frac{\partial u}{\partial\nu}=\frac{1}{r}u_s,\ \ q=\pm\frac{1}{r}\cot\gamma,$$
where $+$ (resp. $-$) occurs if the cylinder is convex (resp. concave).   Then $g_n$ satisfies \eqref{eq-g} and the boundary conditions are now:
\begin{equation}\label{boundary2}
g_n(0)=0,\ g_n'(\beta)\pm\cot\gamma g_n(\beta)=0,
\end{equation}
with $-$ (resp. $+$) if the cylinder is convex (resp. concave). In order to study the stability problem of the cylinder, we distinguish both cases.

\begin{proposition} Under the above conditions,  a concave cylinder is   stable.
\end{proposition}
\begin{proof}
Because the cylinder lies in the wedge, the contact angle $\gamma$ satisfies $0\leq\gamma<\pi/2$ and $\beta<\pi/2-\gamma$. See Fig. \ref{fig4}. We solve \eqref{eq-g}  letting $C= r^2(\frac{1}{r^2}-\frac{n^2\pi^2}{h^2}+\lambda)$ again.
\begin{enumerate}
\item Case $C<0$. Put $c=\sqrt{-C}$. The solution is $g_n(s)=Ae^{cs}+Be^{-cs}$ and the conditions \eqref{boundary2} are equivalent to
$$A+B=0,\ \  Ac(e^{c\beta}+e^{-c\beta})+\cot\gamma B (e^{c\beta}-e^{-c\beta})=0.$$
This says that $B=-A$ and the second equation writes as $c\tan\gamma+\tanh(c\beta)=0$, which is a contradiction because $\tan\gamma\geq 0$ and $c$ and $c\beta$ are positive numbers.
\item Case $C=0$. Then $g_n(s)=As+B$, for some numbers  $A$ and $B$. As $g_n(0)=0$, then $B=0$ and the second equation in \eqref{boundary2} means $A(1+\beta\cot\gamma)=0$, which it is a contradiction again.
\item Case $C>0$. Now $g_n(s)=A\cos(sc)+B\sin(sc)$, where $A,B\in\r$. Since $g_n(0)=0$, then $A=0$. Then other equation in \eqref{boundary2} is $c\cos(c\beta)+\sin(c\beta)\cot\gamma=0$ or equivalently, $c\tan\gamma+\tan(c\beta)=0$. If we see this equation on $c$, $c>0$, this implies that
$c\beta\in(\pi/2,\pi)$ that is, $c>\pi/(2\beta)$. As $\beta<\pi/2$, we have  $c>1$. From the expression of $C$, we have
$$1<c^2=1-\frac{n^2\pi^2 r^2}{h^2}+\lambda r^2,$$
which implies that $\lambda$ is always positive for any value of $h$. In particular,   the cylinder is stable by Lemma \ref{le-eigen}.
\end{enumerate}
\end{proof}
\begin{figure}[hbtp]
\begin{center}
  \includegraphics[width=.8\textwidth]{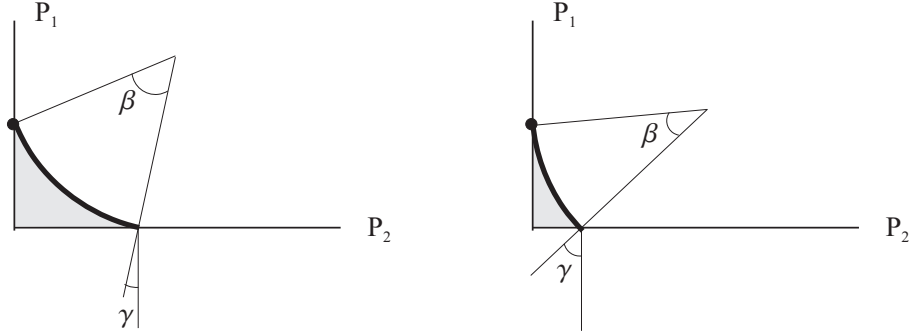}
\end{center}
\caption{Concave  cylinders in a right wedge where the contact angle satisfies $0<\gamma<\pi/2$.}\label{fig4}
\end{figure}

We study a convex cylinder in the  particular case that the contact angle is $\gamma=\pi/2$. See Fig. \ref{fig5}.

\begin{proposition}\label{pr-pi2}   Assume $\gamma=\pi/2$ and let the convex cylinder that makes a contact angle $\gamma=\pi/2$ with $P_2$. If $\beta\leq\pi/2$, then  is stable and if $\beta>\pi/2$, then it is not stable.
\end{proposition}

\begin{proof}
The boundary conditions \eqref{boundary2} write now as $g_n(0)=0$ and $g_n'(\beta)=0$. We distinguish  three cases again.
\begin{enumerate}
\item If $C<0$, $g_n(s)=Ae^{cs}+Be^{-cs}$, where $A,B\in\r$. The boundary conditions imply $g_n=0$ for any $n$: contradiction.
\item If $C=0$, $g_n(s)=As+B$, $A,B\in\r$. The boundary conditions give $g_n=0$, which it is impossible again.
\item If $C>0$, $g_n(s)=A\cos(sc)+B\sin(sc)$, $A,B\in\r$. As $g_n(0)=0$, $A=0$. From the second equation, $\cos(c\beta)=0$, that is, $c\beta=\pi/2+k\pi$, $k\in\mathbb{N}\cup\{0\}$. Then $c^2\geq\pi^2/(4\beta^2)$. If $\beta\leq\pi/2$, we have from the expression of the constant $C$ that
    $$\lambda=\frac{n^2\pi^2}{h^2}+\frac{c^2-1}{r^2}\geq \frac{n^2\pi^2}{h^2}+\frac{\pi^2/(4\beta^2)-1}{r^2}>0,$$
    showing that the surface is stable. If $\pi/2<\beta<\pi$, then $\pi/(4\beta^2)<1$. If we take $k=0$, the number $c^2-1$  in the expression of $\lambda$ in terms of $c^2$ is negative.   Assuming $h$ sufficiently big, we obtain many negative eigenvalues, which shows that the surface is not stable by Lemma \ref{le-eigen}.
\end{enumerate}

\end{proof}
\begin{figure}[hbtp]
\begin{center}
  \includegraphics[width=.8\textwidth]{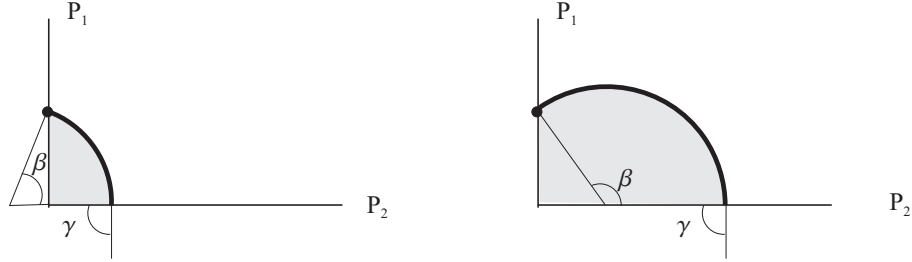}
\end{center}
\caption{Convex cylinders in a right wedge: case $\gamma=\pi/2$.}\label{fig5}
\end{figure}

\begin{proposition}\label{pr-52} Under the setting of this Section, a convex cylinder of length $h>0$ and $\gamma\not=\pi/2$ is stable if and only if the following conditions hold:
\begin{enumerate}
\item $\gamma<\pi/2$ and $e^{2c\beta}\not=(1+c\tan\gamma)/(1-c\tan\gamma)$.
\item $\gamma<\pi/2$ and $\beta\not=\tan\gamma$.
\item $\gamma<\pi/2$, $c\beta<\pi/2$ and $c\tan\gamma-\tan(c\beta)=0$ has  no root for $c\in(0,1)$.
\item $\gamma>\pi/2$, $c\beta>\pi/2$ and $c\tan\gamma-\tan(c\beta)=0$ has  no root for $c\in(0,1)$.
\end{enumerate}
\end{proposition}
\begin{proof}
As in the above proposition, we solve the eigenvalue problem \eqref{eq-eigen}. We point out that in the case that the cylinder is convex, $\gamma$ can take any value in the interval $(0,\pi)$: see Fig. \ref{fig6}.  We use \eqref{uts2} and the   boundary conditions \eqref{boundary2} with the choice of the sign $-$ in the second equation. We analyze all the possibilities according to the sign of the constant $C$.
\begin{enumerate}
\item Case $C<0$. If $c=\sqrt{-C}$, the solution is $g_n(s)=Ae^{cs}+Be^{-cs}$. From $g_n(0)=0$, we deduce $B=-A$, and the second equation of \eqref{boundary2} writes now as $c\tan\gamma-\tanh(c\beta)=0$.
This equation has no solutions if $\gamma>\pi/2$. If $\gamma<\pi/2$,  it is possible the existence of such $c$,  exactly,
$$e^{2c\beta}=\frac{1+c\tan\gamma}{1-c\tan\gamma}.$$
For this value of $c$, the eigenvalues are
$$\lambda=\frac{n^2\pi^2}{h^2}-\frac{c^2+1}{r^2}.$$
If $h$ is sufficiently big, there are many $n$'s such that the corresponding eigenvalue $\lambda$ is negative. This means that the surface is not stable by Lemma \ref{le-eigen}.
\item Case $C=0$. Then $g_n(s)=As+B$ with $B=0$ and $A(1-\cot\gamma \beta)=0$. If $\gamma>\pi/2$, this is not possible. If $\gamma<\pi/2$, then $\beta=\tan\gamma$. In such case, the eigenvalues $\lambda$ are
    $$\lambda=\frac{n^2\pi^2}{h^2}-\frac{1}{r^2}.$$
    Again, if $h$ is sufficiently big, there are many integers $n$ so the corresponding eigenvalue is negative, which shows that the surface is not stable.
\item Case $C>0$. Now $g_n(s)=A\cos(sc)+B\sin(sc)$. Since $g_n(0)=0$, then $A=0$. Then the second equation of \eqref{boundary2} is
$c\cos(c\beta)-\cot\gamma\sin(c\beta)=0$, that is, $c\tan\gamma-\tan(c\beta)=0$. This equation is not solvable if $\gamma<\pi/2$ and $c\beta\geq \pi/2$ or $\gamma>\pi/2$ and $c\beta\leq\pi/2$. In the other cases,
$$\lambda=\frac{n^2\pi^2}{h^2}+\frac{c^2-1}{r^2}$$
and if $c\tan\gamma-\tan(c\beta)=0$ has roots on $c\in (0,1)$, then for $h$ sufficiently big, there are many eigenvalues $\lambda$ that are negative and so, the surface is not stable.
\end{enumerate}

\end{proof}

\begin{figure}[hbtp]
\begin{center}
  \includegraphics[width=.8\textwidth]{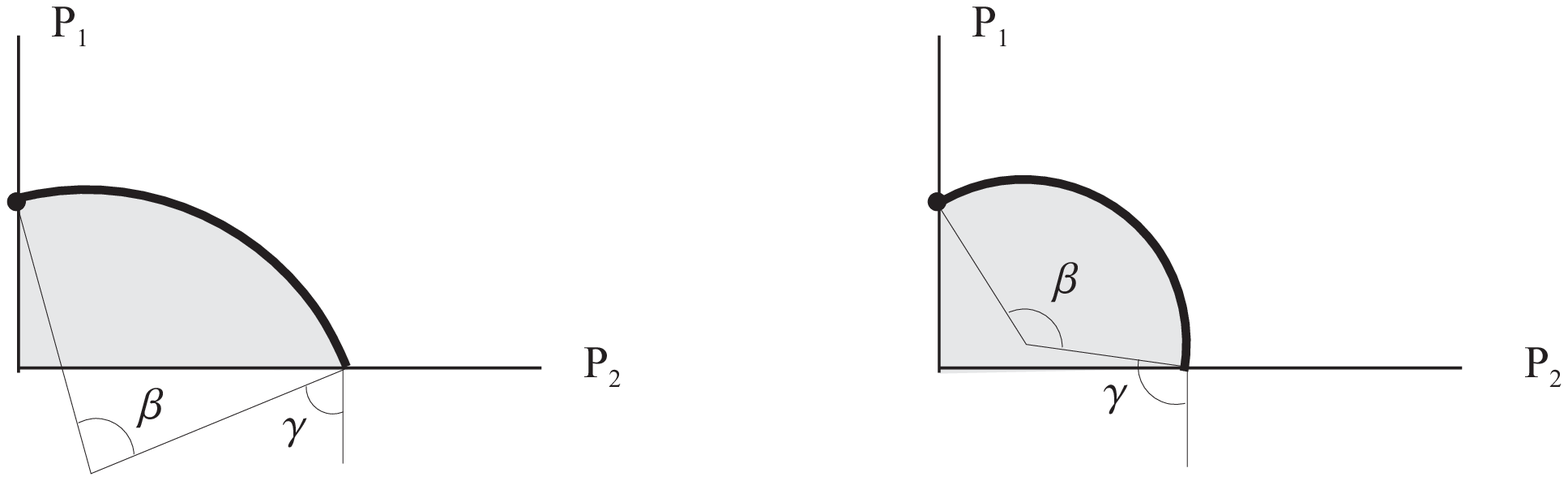}
\end{center}
\caption{Convex  cylinders in a right wedge. On the left, the angle $\gamma$ satisfies $0<\gamma<\pi/2$ and on the right we have $\pi/2<\gamma<\pi$.}\label{fig6}
\end{figure}

\section{Proof of Theorem \ref{t2}}\label{s6}

Let $W$ be a right wedge determined by to orthogonal planes $P_1\cup P_2$. Let $M$ be a surface with non-empty boundary and assume that $\partial M$ has two components, $\partial M=\Gamma_1\cup\Gamma_2$. Let $\phi:M\rightarrow \r^3$ be an immersion whose image lies in the wedge $W$  such that $\phi_{|\Gamma_1}$ is a prescribed curve in the plane $P_1$ and the other one satisfies $\phi_{|\Gamma_2}\subset P_2$.  We consider stationary surfaces of the corresponding variational problem, which leads to that the mean curvature $H$ of the immersion is constant and the angle that makes the surface with the plane $P_2$ is a constant $\gamma$ along the curve $\Gamma_2$. Consider normal admissible variations of $\phi$ given by $\phi+uN$, where $u$ is a smooth function on $M$ that vanishes on $\Gamma_1$. If $V$ in an open of $0\in C^{2,\alpha}_0(M)$, we define $F:V\times\r\rightarrow C^\alpha(M)\times\r$ by
$$F(u,H)=(2(H-H_u),\gamma_u-\gamma),$$
where  $H_u$ is the mean curvature of the immersion and  $\gamma_u$ is the angle that makes the surface $\phi+uN$ with the plane $P_2$: see \cite{vo}. The analogous result of Lemma \ref{le-ko}, item 2, for the eigenvalue $\lambda=0$, is now:

\begin{lemma}[Lemma 3.9 in \cite{vo}]\label{le-h0h1}
The functional $F$ is Fr\'echet differentiable and
$$D_uF(0,H)(v)=(Lv,{\cal B}v), $$
where   $Lv=\Delta v+|\sigma|^2v$ and the operator ${\cal B}$ is
$${\cal B}v=\frac{\partial v}{\partial\nu}-qv\ \ \mbox{on $\Gamma_2$}.$$
A pair of differentiable functions $(\varphi_1,\varphi_2)\in C_0^\infty(M)\times C^\infty(\Gamma_2)$  lies in the image of $D_uF(0,H)$ if and only if for any $u_0\in E_0$,
$$\int_M u_0 \varphi_1\ dM-\int_{\Gamma_2} u_0 \varphi_2\ ds=0.$$
\end{lemma}

Consider the particular case that the contact angle $\gamma$ with the plane $P_2$ is exactly $\gamma=\pi/2$. By Proposition \ref{pr-pi2} we know that if $\beta\leq\pi/2$ the surface is stable and the bifurcation can not take place. Therefore we study the case  $\beta>\pi/2$.

\begin{theorem}\label{t-pi2} For a convex cylinder and in the case $\gamma=\pi/2$,  $\beta>\pi/2$, if we denote
$$T=\frac{4\pi r\beta}{\sqrt{4\beta^2-\pi^2}},$$
then the convex cylinder $C(r,\pi/2)$ bounded by $L_1$ and making an contact angle $\gamma$ with $P_2$ bifurcates in periodic surfaces with period $T$.
\end{theorem}

\begin{proof} We know by Proposition \ref{pr-pi2} that the eigenvalues of \eqref{eq-eigen} occur when $C>0$. In such case the eigenvalues are
\begin{equation}\label{2lambda}
\lambda_{k,n}=\frac{n^2\pi^2}{h^2}+\frac{c^2-1}{r^2}
\end{equation}
with $c\beta=\pi/2+k\pi$, $k\in\mathbb{N}\cup\{0\}$, $n\in\mathbb{Z}$.  If $h=T/2$, the first eigenvalue is  $0$. If $h$ goes from $T/2$ to $T$, the first eigenvalue is negative, but the other ones $\lambda_{k,n}$ are all positive.  It is just   for $h=T$ when the second eigenvalue is $0$.  We show that at this moment there exists a bifurcation point.

We do a similar reasoning as in the proof of Theorem \ref{t1}. Because we look for $T$-periodic surfaces, we take separation of variables with a function $u$ as in \eqref{2uts}. From Proposition \ref{pr-pi2}, the function $g_n$ satisfies $g_n''(s)+c^2 g_n(s)=0$ where now $c^2$ is
$$c^2=r^2\Big(\frac{1}{r^2}-\frac{4n^2\pi^2}{h^2}+\lambda\Big).$$
From \eqref{2lambda} and the value of $T$ given in the statement of the theorem, we have
$$\lambda_{k,n}=\frac{n^2(4\beta^2-\pi^2)+4\beta^2(c^2-1)}{4r^2\beta^2}.$$
Then $\lambda=0$ is an eigenvalue for $k=0$ and $n=1$ ($c=\pi/(2\beta)$).   The corresponding eigenfunction is
$u_{0,1}(t,s)=\sin(cs)\sin(2\pi t/T)$.  We study if the hypothesis  of the bifurcation theorem of Crandall and  Rabinowitz are fulfilled.  We know that $E_0=\mbox{span}\{u_{0,1}\}$ and as a consequence $\mbox{dim}(E_0)=1$. On the other hand, we take
$(\varphi_1,\varphi_2)\in C^{\infty}_{0}(\r/2\pi T\mathbb{Z}\times [0,\beta])\times C^{\infty}_T(\Gamma_2)$ in order to compute the codimension of $\mbox{Im}(D_uF(0,H))$. We know by Lemma \ref{le-h0h1} that $(\varphi_1,\varphi_2)\in \mbox{Im}(D_uF(0,H))$ if and only if $\int_M u_{0,1} \varphi_1\ dM-\int_{\Gamma_2} u_{0,1} \varphi_2\ ds=0$. But
${\cal B}u_{0,1}=0$. Then $(\varphi_1,\varphi_2)\in \mbox{Im}(D_uF(0,H))$ if and only if
$\int_M u_{0,1}\varphi_1\ dM=0$, that is, if it belongs to the orthogonal subspace of $u_{0,1}$.
This shows that  the codimension is $1$. Finally we check that $D_HD_uF(0,H)(u_{0,1})\not\in\mbox{Im}(D_uF(0,H))$. From the definition of $F$ in Lemma \ref{le-h0h1} and \eqref{duh} we have
$$D_HD_uF(0,H)(u_{0,1})=(8H((u_{0,1})_{ss}+u_{0,1}),{\cal B}u_{0,1})=(8H((u_{0,1})_{ss}+u_{0,1}),0).$$
If $(8H((u_{0,1})_{ss}+u_{0,1}),0)\in \mbox{Im}(D_uF(0,H))$, then we would have
\begin{equation}\label{wou}
\int_M 8H((u_{0,1})_{ss}+u_{0,1})u_{0,1}dM=0.
\end{equation}
However
$$\int_M 8H((u_{0,1})_{ss}+u_{0,1})u_{0,1}\ dM=\int_0^\beta\int_0^{T} 8H(1-c^2)\sin^2(cs)\sin^2(\frac{2\pi t}{T})\ dsdt,$$
which it is not zero because $c^2-1\not=0$. This contradicts \eqref{wou}.
\end{proof}

\begin{remark} By the symmetry of solutions given in   Corollary \ref{co}, Theorem \ref{t-pi2} can see as a particular case of Theorem \ref{t1}. In this case, the value $\beta$ corresponds with the angle $\gamma$ in Theorem \ref{t1}, obtaining the same value of period $T$.
\end{remark}

We study the case $\gamma\not=\pi/2$. We know that   a concave cylinder is  stable. Then we pay our attention on a convex cylinder. The study is similar as in the proof of Theorem \ref{t1}. Given a convex cylinder $C(r,\gamma)$, for small wavelengths $h$ the surface is stable. As we increase the value of $h$, we arrive the first value $h_0$ such  that $\lambda_1=0$. We continue increasing $h$. Then the first eigenvalue  is negative, but the next ones are positive until that we arrive to a new value of $h$, namely, $h=T$, such that the second eigenvalue of \eqref{eq-eigen} is $0$. For this value of length for $C(r,\gamma)$ we shall prove that we are under the hypothesis of Theorem of Crandall-Rabinowitz, showing the existence of a bifurcation point. As our solutions will be $T$-periodic, we study the periodic eigenvalue problem \eqref{eq-eigen}. For this, we write $u=u(t,s)$ as in \eqref{2uts}. The functions $g_n$ satisfy $g_n''(s)+C g_n(s)=0$ where
$$C=r^2\Big(\frac{1}{r^2}-\frac{4n^2\pi^2}{T^2}+\lambda\Big)$$
and the boundary conditions are
\begin{equation}\label{boundary4}
g_n(0)=0,\ \ \ g_n'(\beta)-\cot\gamma g_n(\beta)=0.
\end{equation}

\begin{theorem} Assume $\gamma<\pi/2$ and $\beta>\tan\gamma$. Then the   convex cylinder $C(r,\gamma)$ bifurcates.
\end{theorem}
\begin{proof}
Doing a similar computations as in Proposition \ref{pr-52}, we solve $g_n$ depending  on the sign on $C$. If $C=0$, $g_n(s)=As$ with $A(1-\beta\cot\gamma)=0$, which it is a contradiction. If $C>0$,
then $g_n(s)=B\sin(cs)$, with $c\tan\gamma-\tan{(c\beta)}=0$. We see this equation on $c$, where at $c=0$, is zero and it is strictly decreasing by the fact that $\gamma<\pi/2$ and $\beta>\tan\gamma$. Thus the only possibility is that $C<0$. In such case, there is an  eigenvalue $\lambda$ if
$$e^{2c\beta}=\frac{1+c\tan\gamma}{1-c\tan\gamma}.$$
 As  $\beta>\tan\gamma$, there is a unique solution $c$. This is because $\frac{1+c\tan\gamma}{1-c\tan\gamma}$ is an  increasing function on $c$, that goes from $1$ to $\infty $ in the interval $(0,\tan\gamma)$ and from
    $-\infty$ to $-1$ in the interval $(\cot\gamma,\infty)$. On the other hand, the value $e^{2c\beta}$ is increasing on $c$, going from $1$ to $\infty$.  In such case, $0$ is an eigenvalue for
    $$T=\frac{2\pi r}{\sqrt{1+c^2}}.$$
    The corresponding eigenfunction is
    $$u_1(t,s)=g_1(s)\sin(\frac{2\pi t}{T})=(e^{cs}-e^{-cs})\sin(\frac{2\pi t}{T}).$$
 In particular, $\mbox{dim}(E_0)=1$. As in     Theorem \ref{t-pi2}, $(\varphi_1,\varphi_2)\in \mbox{Im}(D_u F(0,H))$ if    $\varphi_1$ is orthogonal to $u_1$, which shows that the codimension of Im$(D_uF(0,H))$ is $1$. As ${\cal B}u_1=0$,
    $D_HD_u F(0,H)(u_1)=(8H(u_1)_{ss}+u_1,0)$ but
    $$\int_M u_1 ( D_H D_uF(0,H)(u_1))\ dM=\int_M 8H(1+c^2)u_1^2\ dM\not=0.$$
    This means that $D_H D_uF(0,H)(u_1)\not\in\mbox{Im}(D_uF(0,H))$. Then Theorem \ref{t-cr} shows that a bifurcation does exist.
\end{proof}

\begin{theorem} Assume $\gamma<\pi/2$ and $\beta=\tan\gamma$. Then the corresponding convex cylinder $C(r,\gamma)$ bifurcates.
\end{theorem}

\begin{proof} We repeat the above arguments. The only possibility to solve  the equation $g_n''(s)+C g_n(s)=0$ with boundary conditions
\eqref{boundary4} is that $C=0$. In such case, the solution is $g_n(s)=As$ and the eigenvalues are
$$\lambda_n=\frac{4n^2\pi^2}{h^2}-\frac{1}{r^2}.$$
Let $T=2\pi r$. Then $\lambda_n=0$ is an eigenvalue of the periodic eigenvalue problem \eqref{eq-eigen} if $n$ takes the value $n=1$. The corresponding eigenspace is
$E_0=\mbox{span}\{u_1\}$, where  $u_1(t,s)=s\sin(2\pi t/T)$. In particular, $\mbox{dim}(E_0)=1$.  Now ${\cal B}u_{1}=0$ and
$$D_HD_uF(0,H)(u_{1})=8H((u_{1})_{ss}+u_{1}),0)=(8Hu_1,0).$$
Using Lemma \ref{le-h0h1}, we have that $(8Hu_1,0)\not\in \mbox{Im}(D_uF(0,H))$ because
$$\int_M 8Hu_1^2\ dM\not=0.$$
\end{proof}

\begin{theorem} Assume that $c\tan\gamma-\tan(c\beta)=0$ has  a  root for $c\in(0,1)$ and either one of the next hypothesis: i) $\gamma<\pi/2$, $c\beta<\pi/2$ or ii) $\gamma>\pi/2$, $c\beta>\pi/2$. Then the convex cylinder $C(r,\gamma)$ bifurcates.
\end{theorem}
\begin{proof} The reasoning in both cases is similar and we only consider the first one.  We know that there is solution of \eqref{eq-g} if $C>0$. In such case the function $g_n$ is $g_n(s)=A\sin(cs)$ with $c\tan\gamma-\tan(c\beta)=0$. This equation has a root for some $c\in (0,1)$. We claim that this solution is unique. For this,   we define the function $\psi(c)=c\tan\gamma-\tan(c\beta)$, which satisfies $\psi(0)=0$, $\psi'(0)>0$ and $\psi$ decreases monotonically as $c\rightarrow \pi/(2\beta)$. This proves that the solution $c$ is unique. A similar reasoning as in the above results provides the value of the period:  $T=2\pi r/\sqrt{1-c^2}$. For this value of $T$, $\lambda=0$ is an eigenvalue whose eigenfunction is  $u_1(t,s)=g_1(s)\sin(2\pi t/T)$ with $g_1(s)=\sin(cs)$. We now study the third hypothesis in Theorem \ref{t-cr}. The value of ${\cal B}u_{1}$ is zero again. Here
$D_HD_uF(0,H)(u_1)=(8H(u_1)_{ss}+u_1,0)$.  If this pair belongs to $\mbox{Im}(D_uF(0,H))$, Lemma \ref{le-h0h1} implies that
$$\int_M u_2(8H(u_2)_{ss}+u_2)\ dM=\int_M 8H(1-c^2)u_2^2\ dM=0.$$
As $c^2-1\not=0$, we obtain a contradiction and thus $D_HD_uF(0,H)(u_1)\not\in \mbox{Im}(D_uF(0,H))$. This proves the result.
\end{proof}

{\it Acknowledgement}. The   author would like to thank Antonio Ros for  his helpful
comments and suggestions in this work. Part of this work has been developed when the author was visiting  the Max Planck Institute of Colloids and Interfaces, at Potsdam in 2010. The author would like to  thank   Prof. Lipowsky and its research group for   his hospitality.



\end{document}